\documentclass[11pt]{amsart}
\usepackage{graphicx}
\usepackage{amsfonts, amssymb, amsmath,amsthm, color, url} 
\usepackage[numbers]{natbib}
\usepackage{epsfig}
\usepackage{ifthen}

\newtheorem{Thm}{Theorem}
\newtheorem{Lem}{Lemma}
\newtheorem{Prop}{Proposition}

\theoremstyle{definition}
\newtheorem{Rem}{Remark}

\newcommand{\comment}[1]{}

\def\indd#1{{\bf 1}_{\{#1\}}}

\def\indn#1{\{#1_n\}_{n\in\N}}

\newcommand{\proba}{\mathbb P}
\newcommand{\esp}{{\mathbb E}}

\newcommand{\defe}{\mathrel{\mathop:}=}
\newcommand{\inv}{^{-1}}
\newcommand{\cov}{{\rm{Cov}}}

\def\B{{\mathbb B}}

\newcommand{\eqnh}{\begin{eqnarray*}}
\newcommand{\eqne}{\end{eqnarray*}}
\newcommand{\eqnhn}{\begin{eqnarray}}
\newcommand{\eqnen}{\end{eqnarray}}
\newcommand{\equh}{\begin{equation}}
\newcommand{\eque}{\end{equation}}

\def\summ#1#2#3{\sum_{#1 = #2}^{#3}}

\def\sif#1#2{\sum_{#1=#2}^\infty}

\def\bcupp#1#2#3{\bigcup_{#1=#2}^{#3}}

\newcommand{\eqd}{\stackrel{\rm d}{=}}

\def\topp#1{^{(#1)}}

\def\ccbb#1{\left\{#1\right\}}

\def\pp#1{\left(#1\right)} 
 
\def\bb#1{\left[#1\right]}

\def\mmid{\;\middle\vert\;}

\def\floor#1{\left\lfloor #1 \right\rfloor}

\def\vv#1{{\bf #1}}

\def\d{{\rm d}}

\def\B{{\mathbb B}}

\def\mand{\mbox{ and }}
\def\qmand{\quad\mbox{ and }\quad}

\def\qmwith{\quad\mbox{ with }\quad}
\def\mfa{\mbox{ for all }}

\def\mmas{\mbox{ as }}

\def\wt#1{\widetilde{#1}}

\def\weakto{\Rightarrow}

\def\limn{\lim_{n\to\infty}}

\def\limsupn{\limsup_{n\to\infty}}

\def\R{{\mathbb R}}

\def\N{{\mathbb N}}

\begin{document}\sloppy
\title[Maximum process of f{B}m with shot noise]{Convergence to the maximum process of a fractional Brownian motion with shot noise}

\author{Yizao Wang}
\address
{
Yizao Wang,
Department of Mathematical Sciences,
University of Cincinnati,
2815 Commons Way, ML--0025,
Cincinnati, OH, 45221-0025.
}
\email{yizao.wang@uc.edu}

\begin{abstract}
We consider the maximum process of a random walk with additive independent noise in form of $\max_{i=1,\dots,n}(S_i+Y_i)$. The random walk may have dependent increments, but its sample path is assumed to converge weakly to a fractional Brownian motion. When the largest noise has the same order as the maximal displacement of the random walk, we establish an invariance principle for the maximum process in the Skorohod topology. The limiting process is the maximum process of the fractional Brownian notion with shot noise generated by Poisson point processes.   \end{abstract}


\maketitle

\section{Introduction}
Let $\indn X$ and $\indn Y$ be two independent sequences of random variables. Write $S_0 = 0, S_n = X_1+\cdots+X_n$ and we are interested in the asymptotic behavior of the maximum process $M_0 = 0$,
\[
M_n = \max_{i=1,\dots,n}(S_i+Y_i), n\in\N.
\]
We view $\indn S$ as a random walk and $\indn Y$ perturbations (or noise).
We allow dependence between steps of random walk, while  the perturbations $\indn Y$ are assumed to be independent and identically distributed (i.i.d.).

We consider the general framework that the sample path of the random walk $\{S_{\floor{nt}}\}_{t\in[0,1]}$ without perturbation converges weakly to a stochastic process, after appropriate normalization. Such results are referred to as invariance principles (and/or functional central limit theorems) in the literature. This includes Donsker's theorem \citep{donsker51invariance} which states that when $\indn X$ are i.i.d.~with zero mean and unit variance, the sample path converges weakly to a standard Brownian motion. 

More generally when $\indn X$ is stationary, weak convergence to fractional Brownian motions has been extensively investigated. 
Throughout, we assume the following invariance principle to hold for $S_n$: 
\equh\label{eq:SnWIP}
\ccbb{\frac{S_{\floor{tn}}}{n^H}}_{t\in[0,1]} \weakto\ccbb{\B^H_t}_{t\in[0,1]}
\eque
in the space of c\`adl\`ag functions $D[0,1]$, 
where $\B^H = \{\B^H_t\}_{t\in[0,1]}$ is the fractional Brownian motion with Hurst index $H\in(0,1)$. This is the zero mean Gaussian process with covariance function $\cov(\B_s,\B_t) = 2^{-1}(s^{2H}+t^{2H} + |s-t|^{2H})$, $s,t\geq0$.
An extensively studied model of $\indn X$ leading to the invariance principle~\eqref{eq:SnWIP} is stationary linear process. The seminal work of \citet{taqqu75weak} addressed the case when innovations of the linear processes are independent; for dependent innovations, we refer to the recent result of \citet{dedecker11invariance} and references therein, among many others under various dependence assumptions.

We are interested in the behavior of the maximum process of $\indn S$ perturbed with $\indn Y$, given that the invariance principle~\eqref{eq:SnWIP} holds for some $H\in(0,1)$. It is clear  that the only non-trivial case is when the tail distribution of $Y_i$ satisfies, for some $\kappa>0$,
\equh\label{eq:RV}
\proba(Y_i>x)\sim \kappa x^{-{1/H}} \mmas x\to\infty.
\eque
Indeed, in this case, 
\[
\limn\proba\pp{\max_{i=1,\dots,n}\frac{Y_i}{n^H}\leq x} = \exp(-\kappa x^{-1/H}), x>0,
\]
and the normalization $n^H$ gives non-degenerate distributional limit for both $S_n$ and $Y_n$. Otherwise, 
if $\proba(Y_i>x)\sim \kappa x^{-\beta}$ for some $\beta\neq1/H$, then either $\max_{i=1,\dots,n}S_i$ or $\max_{i=1,\dots,n}Y_i$ will dominate after appropriate normalization, and the result is immediate. 

The main result of this paper is an invariance principle in form of
\equh\label{eq:WIP}
\ccbb{\frac{M_{\floor{nt}}}{n^H}}_{t\in[0,1]}\weakto \{Z^H_t\}_{t\in[0,1]}, \mbox{ in $D[0,1]$}.
\eque
The limiting process $\{Z^H_t\}_{t\in\R_+}$ is defined as follows. Let $\B^H = \{\B^H_t\}_{t\in\R_+}$ be a fractional Brownian motion, and let $\{\eta_i,U_i\}_{i\in\N}$ be a Poisson point process on $\R_+\times\R_+$ with intensity $H\inv x^{-1-1/H}\d x \d u$, in the same probability space as and independent of $\{\B^H_t\}_{t\in\R_+}$. Throughout, we assume $\B^H$ has continuous sample path \citep{samorodnitsky94stable}. Then, the limiting process is defined as
\equh\label{eq:Z}
Z^H_t(\kappa) \defe \sup_{s\in[0,t]}\pp{\B^H_s + \kappa^H\sif i1\eta_i\indd{U_i = s}}, t\in\R_+.
\eque
In the sequel, we fix $\kappa\in(0,\infty), H\in(0,1)$ and write $Z^H\equiv\{Z^H_t\}_{t\in\R_+}\equiv \{Z^H_t(\kappa)\}_{t\in\R_+}$ for the sake of simplicity. 
The main result is the following.
\begin{Thm}\label{thm:1}
Let $\indn X$ and $\indn Y$ be two independent sequences of random variables. Suppose that $\indn X$ satisfy~\eqref{eq:SnWIP} and $\indn Y$ are i.i.d.~random variables satisfying~\eqref{eq:RV}. Then, the invariance principle~\eqref{eq:WIP} holds in the space of $D[0,1]$ equipped with the Skorohod $J$-1 topology. 
\end{Thm}
We refer to $Z^H$ in~\eqref{eq:Z} as the maximum process of {\it fractional Brownian motion with shot noise}. 
To the best of our knowledge, the limiting process $Z^H$ is new. 

Our motivation originally came from 
\citet{hitczenko11renorming} in the study of perpetuities, who raised an open question in our framework with $S_n$ converging weakly to a standard Brownian motion (the open question was actually on the marginal distribution of $M_n$).  Theorem~\ref{thm:1} indicates that the conjectured limiting object in \citet[Remarks, p.~889]{hitczenko11renorming} is incorrect.

The model $\{S_n+Y_n\}_{n\in\N}$ has been seen in the literature as the {\it perturbed random walk}, see for example \citet{araman06tail} and \citet{alsmeyer14power}, and references therein for other motivations and applications. Here we take a different aspect, while most of the results in the literature assume negative drift of the random walk $\indn S$. Moreover, the same name {\it perturbed random walk} has been used for another model of self-interacting random walk, see \citet{davis96weak} and \citet{perman97perturbed}. The corresponding limiting process, the so-called {\it perturbed Brownian motion}, has also been characterized and investigated. We choose not to use the name {\it perturbed fractional Brownian motion} for our process $Z^H$.

Theorem~\ref{thm:1} is first proved in the case when $\indn Y$ are non-negative. For this part the proof is essentially an application of continuous mapping theorem combined with a truncation argument. To do so, recall the invariance principle for $S_n$~\eqref{eq:SnWIP} and also the weak convergence for order statistics of $\indn Y$. Indeed, under~\eqref{eq:RV},
the order statistics $Y_{1,n}\geq Y_{2,n}\geq\cdots\geq Y_{n,n}$ of $Y_1,\dots,Y_n$ satisfy
\equh\label{eq:lepage}
\ccbb{\pp{\frac{Y_{i,n}}{n^H},\frac in}}_{i=1,\dots,n} \weakto \ccbb{(\kappa^H\eta_n,U_n)}_{n\in\N}
\eque
in the space of Radon point measures on $(0,\infty)\times(0,1)$,
where $(\eta,\vv U) = \{(\eta_n,U_n)\}_{n\in\N}$ as before is a Poisson point process on $\R_+\times(0,1)$ with intensity measure $H\inv x^{-1-1/H}\d x\d u$. See the seminal work of \citet{lepage81convergence} and also  \citet[Corollary 4.19]{resnick87extreme}. 
If one could represent $Z^H$ as the image of a continuous function evaluated at $\B^H$ and $(\eta,\vv U)$, then the result would be immediate. In the proof, however, a continuous mapping is constructed for truncated versions of $Z^H$, and we proceed by an approximation argument truncating small perturbations and noise. 
For the general case when $\indn Y$ may be negative, we show that the distribution of $Y_i\indd{Y_i<0}$ does not effect the maximum process in the limit by a coupling argument.

The paper is organized as follows. The finite-dimensional distributions of the limiting process $Z^H$ is given in Section~\ref{sec:Z}. The proof of Theorem~\ref{thm:1} with non-negative perturbation is given in Section~\ref{sec:thm1}. The case with negative perturbation is addressed in Section~\ref{sec:thm1negative}.\smallskip

{\bf Acknowledgement} The author thanks Jacek Weso\l owski for helpful discussions.

\section{Maximum process of Fractional Brownian motion with shot noise}\label{sec:Z}
We first give an explicit formula of the finite-dimensional distributions of $Z^H$. Throughout, $H\in(0,1)$ and $\kappa\in(0,\infty)$. By convention, $\exp(-\infty) = 0$. 
\begin{Prop}\label{prop:1}
Consider the process $Z^H$ defined in~\eqref{eq:Z}. For $0=t_0<t_1<\dots<t_d<\infty$ and $x_1,\dots,x_d\in\R$, 
\begin{multline}\label{eq:Zfdd}
\proba(Z^H_{t_1}\leq x_1,\dots,Z^H_{t_d}\leq x_d) \\
= \esp{\exp\ccbb{-\summ q1d\int_{t_{q-1}}^{t_q}\frac\kappa{(\min_{j=q}^dx_j-\B^H_t)_+^{1/H}}\d t}}.
\end{multline}
In particular, $Z^H$ is self-similar:
\[
\{Z_{at}^H\}_{t\in\R_+} \eqd a^H\{Z^H_t\}_{t\in\R_+}, \mfa a>0.
\]
\end{Prop}

\begin{proof}
The formula follows from the definition of Poisson point processes. Conditioning on the $\sigma$-algebra generated by $\B^H$, consider the region
\[
E = \bigcup_{q=1}^d\ccbb{(x,s)\in\R_+\times(0,t_q]: \B^H_s+x> x_q}.
\]
Clearly, 
\[
\proba(Z^H_{t_1}\leq x_1,\dots,Z^H_{t_d}\leq x_d) = \esp \bb{\proba\pp{\{\kappa^H\eta_i,U_i\}_{i\in\N}\cap E = \emptyset\mid\B^H}}.
\]
It remains to notice that conditioning on $\B^H$, the inner probability above is of the event of a Poisson random variable equal to zero, and its parameter equals the integration of the intensity over the region $E$ (depending on $\B^H$). Observe that
\begin{multline*}
E = \bcupp q1d\bcupp j1q\ccbb{(x,s)\in\R_+\times(t_{j-1},t_j]:\B_s^H + x >x_q}\\
 = \bcupp j1d\ccbb{(x,s)\in\R_+\times(t_{j-1},t_j]:\B_s^H + x >\min_{q=j,\dots,d}x_q} \mbox{ (disjoint).}
\end{multline*}

Thus, the parameter of the Poisson random variable equals
\begin{multline*}
\esp\pp{\int_E\kappa H\inv  x^{-1-1/H}\d x\d t\mmid \B^H} \\
= \summ j1d\int_{t_{j-1}}^{t_j}\int_{\kappa^{-H}(\min_{q=j,\dots,d}x_q - \B_s^H)_+}^\infty H\inv x^{-1-1/H}\d x\d s\\
= \summ j1d\int_{t_{j-1}}^{t_j}\frac\kappa{(\min_{q=j,\dots,d}x_q - \B_s^H)_+^{1/H}}\d s.
\end{multline*}
This formal calculation has a seemingly issue: some the integrations may equal infinity. This happens when
\equh\label{eq:W}
W \defe \max_{j=1,\dots,d}\pp{\sup_{s\in[t_{j-1},t_j]}\B_s^H-\min_{q=j,\dots,d}x_q} > 0.
\eque
However, this case does not cause any problem, as conditioning on $\B^H$ and the event above,
the event $\{\{\kappa^H\eta_i,U_i\}_{i\in\N}\cap E = \emptyset\}$ has probability zero, which is the same as $\exp(-\infty)$. Indeed, recall that $\indn\eta$ has a cluster point at $0$, and $W>0$ implies that $E\supset(0,\infty)\times[\tau_1,\tau_2]$ for some non-degenerate interval $[\tau_1,\tau_2]\subset(0,1)$. On the other hand,
when $W<0$, the integrations are all well defined with finite values. The only remaining case is when $W = 0$; 
this event, however, has probability zero and can thus be ignored.
We have thus proved
\begin{multline*}
\proba\pp{\{\kappa^H\eta_i,U_i\}_{i\in\N}\cap E = \emptyset\mid\B^H} \\
= \exp\ccbb{-\summ j1d\int_{t_{j-1}}^{t_j}\frac\kappa{(\min_{q=j,\dots,d}x_q - \B_s^H)_+^{1/H}}\d s}.
\end{multline*}
Taking expectation completes the proof of~\eqref{eq:Zfdd}.

To prove the self-similar property, observe that 
\begin{multline*}
\proba(Z^H_{at_1}\leq x_1,\dots,Z^H_{at_d}\leq x_d) \\
= \esp{\exp\ccbb{-\summ q1d\int_{at_{q-1}}^{at_q}\frac\kappa{(\min_{j=q}^dx_j-\B^H_t)_+^{1/H}}\d t}}\\
= \esp{\exp\ccbb{-\summ q1d\int_{at_{q-1}}^{at_q}\frac\kappa{(\min_{j=q}^dx_j-a^{H}\B^H_{t/a})_+^{1/H}}\d t}},
\end{multline*}
where the last step we used the self-similar property of the fractional Brownian motion. By change of variables, the expression equals
\begin{multline*}
\esp{\exp\ccbb{-\summ q1d\int_{t_{q-1}}^{t_q}\frac\kappa{(a^{-H}\min_{j=q}^dx_j-\B^H_{t})_+^{1/H}}\d t}}\\
= \proba(a^HZ^H_{t_1}\leq x_1,\dots,a^HZ^H_{t_d}\leq x_d).
\end{multline*}
This completes the proof.
\end{proof}

The next result focuses on the marginal distribution of $Z_1^H$, denoted by
\equh\label{eq:psi}
\Psi_H(x) = \proba(Z_1^H\leq x) = \esp\exp\ccbb{-\int_0^1\frac\kappa{(x-\B^H_t)_+^{1/H}}\d t}, x\in\R.
\eque
\begin{Prop}\label{prop:2}
For $H\in(0,1)$, $\Psi_H(x) = 0 \mbox{ for all } x\leq 0$, 
$\Psi_H(x)$ is continuous and strictly increasing on $(0,\infty)$, and $\lim_{x\to\infty}\Psi_H(x) = 1$.
\end{Prop}
\begin{proof}
First, observe that by self-similarity, for $x>0$,
\begin{multline*}
\Psi_H(x) = \esp\exp\ccbb{-\int_0^1\frac1{x^{1/H}(1-\B^H_{t/x^{1/H}})^{1/H}_+}\d t} \\
= \esp\exp\ccbb{-\int_0^{x^{-1/H}}\frac1{(1-\B^H_{t})^{1/H}_+}\d t}.
\end{multline*}
From the above it is easy to see that for $x>0$, $\Psi_H(x)$ is strictly increasing, continuous, $\Psi_H(x)<1$, and $\lim_{x\to\infty}\Psi_H(x) = 1$. 

Next, observe that we can also write
\equh\label{eq:psi1}
\Psi_H(x) = \esp\ccbb{\exp\pp{-\int_0^1\frac1{(x-\B^H_t)_+^{1/H}}\d t}\indd{\sup_{0\leq t\leq 1}\B_t^H<x}}.
\eque
It follows from~\eqref{eq:psi1} that $\lim_{x\downarrow 0}\Psi_H(x) = 0$ and $\Psi_H(x) = 0$ for all $x<0$. To see~\eqref{eq:psi1}, it suffices to observe that (i) for $\proba$-almost all $\omega$ in the set $\{\sup_{0\leq t\leq1}\B_t^H> x\}$, the set $\{t\in[0,1]: \B_t(\omega)^H> x\}$ has strictly positive Lebesgue measure, by continuity of the sample path of fractional Brownian motions, so the exponential function equals zero, and (ii) $\proba(\sup_{0\leq t\leq 1}\B^H_t = x) = 0$. 
\end{proof}

We conclude this section by two remarks on $Z^H$. The first remark sheds light on why the negative values of $\indn Y$ do not have an effect in the limiting process.
\begin{Rem}\label{rem:negative}
One has an equivalent definition of $Z^H$  involving negative shot noise. It is natural to consider negative shot noise as with appropriate assumption on the lower tail of $\indn Y$, weak convergence of Radon point measures on $\R\setminus\{0\}\times(0,1)$ can be established. Namely, if 
\equh\label{eq:negative}
\proba(|Y_1|>x) \sim \kappa_0x^{-1/H} \mand \lim_{x\to\infty}\frac{\proba(Y_1>x)}{|Y_1|>x} = \theta = 1-\lim_{x\to\infty}\frac{\proba(Y_1<-x)}{|Y_1|>x}
\eque
for some $\kappa_0\in(0,\infty),\theta\in[0,1]$,~\eqref{eq:lepage} becomes
\[
\ccbb{\pp{\frac{Y_{i,n}}{n^H},\frac in}}_{i=1,\dots,n} \weakto \ccbb{(\kappa_0^H\epsilon_n\eta_n,U_n)}_{n\in\N}
\]
where $\indn\epsilon$ are i.i.d.~random variables independent from $\eta$ and $\vv U$, with law $\proba(\epsilon_1 = 1) = \theta = 1-\proba(\epsilon_1 = -1)$.

In this way, one would naturally define $Z^H$ in~\eqref{eq:Z} as 
\equh\label{eq:Z2}
Z^H_t = \sup_{s\in[0,t]}\pp{\B_s^H+\kappa^H_0 \sif i1\epsilon_i\eta_i\indd{U_i = s}}.
\eque
In view of the discussion after~\eqref{eq:W}, the negative shot noise (those $\eta_i$s with $\epsilon_i = -1$) has no effect in the distribution: when $W<0$ then all negative shot noise do not have an effect, and as long as $W>0$ the conditional probability becomes zero. 
Actually, to match~\eqref{eq:Z2} with~\eqref{eq:Z}, it suffices to  take $\kappa = \kappa_0\theta$ for $\theta\in(0,1]$, and the same calculation in the proof of Proposition~\ref{prop:1} goes through. 

Since the negative noise has no effect on the limiting process, in establishing  Theorem~\ref{thm:1}, instead of~\eqref{eq:negative} no assumption on the distribution of negative values of $Y_i$ is needed at all. 
\end{Rem}
The next remark provides an alternative view of $Z^H$ as a generalization of the extremal processes.
\begin{Rem}
Let $(\eta,\vv U)$ be as before. Then,
\[
V_t = \sup_{i\in\N}\eta_i\indd{U_i\leq t}, t\in\R_+
\]
defines a standard $1/H$-Fr\'echet extremal process, and the law of this process is completely determined by $(\eta,\vv U)$. The extremal process $\{V_t\}_{t\in[0,1]}$ is the limiting process of 
\[
\frac1{\kappa^Hn^H}\ccbb{\max_{i=0,\dots,\floor{nt}}Y_i}_{t\in[0,1]}.
\]
This result was first established by \citet{dwass64extremal} and \citet{lamperti64extreme}, although the point process representation of $V$ and the convergence was first introduced by \citet{pickands71twodimensional}, which has become a standard tool in studying extremes \citep{resnick87extreme}.

Similarly, observe that by the definition of $Z^H$ in~\eqref{eq:Z} and the path-continuity of fractional Brownian motions, 
\[
Z_t^H = \sup_{i\in\N}\pp{\B^H_{U_i}+\kappa^H\eta_i}\indd{U_i\leq t}, t\in\R_+.
\]
Therefore, $Z^H$  can be viewed as the maximum process of the process obtained by gluing shot noise $\indn\eta$ at random locations $\indn U$ to a fractional Brownian motion. This time, the shot noise $(\eta,\vv U)$ perturbed by the factional Brownian motion $\B^H$, namely $\{(\B_{U_n}^H+\eta_i,U_n)\}_{n\in\N}$, completely determine $Z^H$. 

This glued process cannot be studied through finite-dimensional distributions, as all the noise cannot be characterized for fixed $t$. Instead, a general framework is to view extremal processes and $Z^H$ here as random sup measures, as summarized in \citet{obrien90stationary}. We focus on the weak convergence in $D[0,1]$ and therefore do not pursue this direction here.
\end{Rem}
\section{Proof of Theorem~\ref{thm:1}, non-negative perturbation}\label{sec:thm1}
In this section, we prove Theorem~\ref{thm:1} under the additional assumption  that $\proba(Y_i\geq 0) = 1$.
We are interested in
\[
Z_{n,t} = \frac{M_{\floor {nt}}}{n^H}, t\in[0,1] \qmwith M_i = \max_{j=0,\dots,i}(S_j+Y_j), 
\]
and we write $S_0 = Y_0 = 0$ for convenience.
Recall the notation $\{Y_{i,n}\}_{i=1,\dots,n}$ for order statistics of $Y_1,\dots,Y_n$.
For each $k\in\N, n\in\N, n\geq k$,  introduce the truncated approximations of $Z_{n,t}$ and $Z_t^{H}$ by
\[
Z_{n,t}\topp k = \max_{i=0,\dots,\floor{nt}}\frac{S_i+Y_{i}\indd{Y_i\geq Y_{k,n}}}{n^{H}}
\]
and
\[
Z_t^{H,(k)} = \sup_{s\leq t}\pp{\B^H_s+\summ i1k \eta_i\indd{U_i = s}}.
\]
That is, for the random walk we ignore the effect of all but the $k$-largest perturbations, and for the limiting process $Z^H$ we ignore all but the $k$-largest shot noise $\{\eta_i\}_{i=1,\dots,k}$.

By \citet[Theorem 2]{dehling09new}, to prove the desired result it suffices to show
\equh\label{eq:approx1}
\lim_{k\to\infty}\limsupn \proba\pp{\sup_{t\in[0,1]}|Z_{n,t}\topp k - Z_{n,t}|>\epsilon} = 0 \mfa \epsilon>0,
\eque
and
\equh\label{eq:k}
\ccbb{Z_{n,t}\topp k}_{t\in[0,1]}\weakto \ccbb{Z_t^{H,(k)}}_{t\in[0,1]} \mbox{ in $D[0,1]$} \mfa k\in\N.
\eque
The first condition~\eqref{eq:approx1} is easy to verify. Observe that
\[
\sup_{t\in[0,1]}|Z_{n,t}-Z_{n,t}\topp k|\leq \frac{Y_{k+1,n}}{n^{H}}.
\]
Now by~\eqref{eq:lepage}, $Y_{k+1,n}/n^H\weakto \eta_{k+1}$ and thus~\eqref{eq:approx1} follows.

It remains to prove~\eqref{eq:k}, and we apply the continuous mapping theorem \citep{billingsley99convergence}. Consider the metric spaces $D[0,1]$ and $M_p((0,\infty)\times(0,1))$, the latter of which is the space of Radon point measures on $(0,\infty)\times(0,1)$. For more background on Radon point measures, we refer to \citet{resnick87extreme}. Denote the elements in these spaces by $x = \{x_t\}_{t\in[0,1]}$ and $(y,u) = \{y_n,u_n\}_{n\in\N}$ and assume $y_n\geq y_{n+1}, n\in\N$. Consider the mapping 
\[
\Psi\topp k: D[0,1]\times M_p((0,\infty)\times(0,1))\to D[0,1]
\]
defined by
\[
\Psi\topp k(x,y,u) = \{h\topp k_{x,y,u}(t)\}_{t\in[0,1]}
\]
with
\[
h\topp k_{x,y,u}(t) = \sup_{s\leq t}\pp{x_s + \summ i1k y_i\indd{u_i=s}}.
\]
Now, write $B_n^H \equiv \{S_{\floor{nt}/n^H}\}_{t\in[0,1]}$, $\vv Y_n = (Y_{1,n},\dots,Y_{n,n},0,\dots)$ and $ {\vv U}_n = (1/n,\dots,1,0,\dots)$. It follows that
\[
\{Z_{n,t}\topp k\}_{t\in[0,1]} = \Psi\topp k\pp{B^H_n,\frac{\vv Y_n}{n^H}, {\vv U}_n},
\]
and for $\eta = \indn\eta, \vv U = \indn U$,
\[
Z^{H,(k)} = \Psi\topp k\pp{\B^H,\kappa^H\eta,\vv U}.
\]
Recall that $B_n^H\weakto \B^H$ and $(\vv Y_n/n^H,\vv U_n)\weakto (\kappa^H\eta,\vv U)$ (\eqref{eq:SnWIP} and \eqref{eq:lepage} respectively).
Therefore, the desired result follows from the continuous mapping theorem, if one can show that $\Psi\topp k$ is continuous on a subset of $D[0,1]\times M_p((0,\infty)\times(0,1))$ with probability one (the probability induced by $(\B^H,\eta,\vv U)$). Indeed, it suffices to focus on the subset
\[
\Gamma = C[0,1]\times M_p((0,\infty)\times(0,1))\cap\{(x,y,u):\indn u \mbox{ pairwise different}\},
\]
as $\proba((\B^H,\eta,\vv U)\in \Gamma) = 1$. Here and in the sequel, $C[0,1]$ is the space of continuous functions indexed by $[0,1]$, equipped with the supremum metric. 
It remains to prove the following lemma to complete the proof of Theorem~\ref{thm:1}.
\begin{Lem}$\Psi\topp k$ is continuous from $\Gamma$ to $D[0,1]$.
\end{Lem}
\begin{proof}
We first focus on $k=1$. Consider $\{(x\topp n,y\topp n, u\topp n)\}_{n\in\N}\subset \Gamma$ and $(x\topp n,y\topp n,u\topp n)\to (x,y,u)$ as $n\to\infty$. We show 
\equh\label{eq:convergence}
f\topp n \defe \Psi\topp1(x\topp n,y\topp n,u\topp n)\to f \defe \Psi\topp1(x,y,u) \mbox{ in }D[0,1] \mmas n\to\infty.
\eque
It is useful to view $f$ as the pointwise maximum function of the maximum process $m$ of $x$ and a step function as follows (and similarly for $f\topp n$). 
Let
\[
m\topp n_t = \sup_{s\in[0,t]}x_s\topp n \qmand m_t = \sup_{s\in[0,t]}x_s
\]
denote the maximum processes of $x\topp n$ and $x$, respectively. By assumption, $m\topp n$ and $m$ are both in $C[0,1]$. Then,
\equh\label{eq:bump}
f_t = m_t \vee \pp{j_1\indd{u_1\leq t}} \qmwith j_1 = x_{u_1} + y_1,\quad t\geq0.
\eque
In other words, $f$ can be viewed as $m$ lifted up to $j_1$ over an interval $[u_1,v_1]$, where $j_1$ is the height of the shot noise $y_1$ at $u_1$ lifted up by $x_{u_1}$ (the height of $x$ at $u_1$) and $v_1 = 1\wedge\inf\{s\geq 0:m_s\geq j_1\}$. Define similarly 
\equh\label{eq:fnt}
f_t\topp n = m_t\topp n\vee \pp{j_1\topp n\indd{u_1\topp n\leq t}} \qmwith j_1\topp n = x_{u_1\topp n}\topp n + y_1\topp n, \quad t\geq0.
\eque
Remark also that since $x\topp n\to x$ in $C[0,1]$, it follows that 
\equh\label{eq:m}
m\topp n\to m \mbox{ in } C[0,1].
\eque

Now to show~\eqref{eq:convergence}, we divide the domain of the functions in $D[0,1]$ into three parts, namely $[0,u_1-\epsilon],[u_1-\epsilon,u_1+\epsilon]$ and $[u_1+\epsilon,1]$, for some $\epsilon>0$ small enough. We investigate the convergence of functions in $D[0,1]$ on these sub-intervals respectively.
Recall that the Skorohod metric for $D[a,b]$ is given by, for $0\leq a<b\leq1$, $x,\wt x\in D[a,b]$,
\equh\label{eq:skorohod}
d_{a,b}(x,\wt x) = \inf_{\lambda\in\Lambda_{a,b}}\ccbb{\pp{\sup_{t\in[a,b]}|\lambda(t) - t|}\vee\pp{\sup_{t\in[a,b]}|x(t) - \wt x\circ\lambda(t)|}}, 
\eque
with 
\begin{align*}
\Lambda_{a,b} = \{\lambda: &\  [a,b]\to[a,b]:\lambda(a) = a,\lambda(b) = b, \\
&\ \mbox{continuous and strictly increasing.}\}
\end{align*}
More generally for $x,\wt x\in D[0,1]$ and $0\leq a<b\leq 1$, for the sake of simplicity, we write $d_{a,b}(x,\wt x) \defe d_{a,b}(r_{a,b}(x),r_{a,b}(\wt x))$, where
$r_{a,b}(x) = \{x_t\}_{t\in[a,b]}$
denote the restriction of $x$ to $D[a,b]$. 
Observe that to show $f\topp n\to f$ in $D[0,1]$, it suffices to show
\eqnhn
& & \limn d_{0,u_1-\epsilon}(f\topp n,f) = 0\label{eq:r1}\\
& & \lim_{\epsilon\downarrow0}\limsupn d_{u_1-\epsilon,u_1+\epsilon}(f\topp n,f) = 0\label{eq:r2}\\
& & \limn d_{u_1+\epsilon,1}(f\topp n,f) = 0\label{eq:r3}.
\eqnen
For more background on convergence in Skorohod metric, we refer to \citet[Chapter 4]{billingsley99convergence} and \citet[Chapter 4.4.1]{resnick87extreme}.

To show~\eqref{eq:r1}, remark that since $u\topp n\to u$, for $n$ large enough so that $u\topp n_1>u_1-\epsilon$, it follows that
\[
r_{0,u_1-\epsilon}(f\topp n) = r_{0,u_1-\epsilon}(m\topp n) \qmand r_{0,u_1-\epsilon}(f) = r_{0,u_1-\epsilon}(m).
\]
Therefore~\eqref{eq:r1} follows from~\eqref{eq:m}.

To show~\eqref{eq:r3}, observe that for $n$ large enough so that $u_1\topp n<u_1+\epsilon$,
\[
r_{u_1+\epsilon,1}(f\topp n) = \ccbb{j\topp n_1\vee m_s\topp n}_{s\in[u_1+\epsilon,1]}  \qmwith j\topp n_1 = x_{u_1\topp n}\topp n + y_1\topp n,
\]
which is the pointwise maximum function of a constant function and a continuous function. Note also that $r_{u_1+\epsilon,1}(f) = \{j_1\vee m_t\}_{t\in[u_1+\epsilon,1]}$. Observe that
\equh\label{eq:jump}
\limn j\topp n_1 = j_1.
\eque
Since for general continuous functions, $g\topp n\to g$ and $h\topp n\to h$ in $C[a,b]$ imply $g\topp n\vee h\topp n\to g\vee h$ in $C[a,b]$,~\eqref{eq:r3} follows.

At last to show~\eqref{eq:r2}, consider $n$ large enough so that $u_1\topp n\in(u_1-\epsilon,u_1+\epsilon)$. There are three cases to be discussed here. First assume $j_1>m_{u_1}$. This is the case that the shot noise at $u_1$ creates a discontinuity in the path by up-lifting to $j_1$. In this case, since $m_t$ is continuous in $t$, choose $\epsilon>0$ small enough so that $j_1>m_{u_1+\epsilon}$ and take $n$ large enough so that 
\[
j_1\topp n>m_{u_1+\epsilon}\topp n.
\]
This can be done because of~\eqref{eq:jump} and~\eqref{eq:m}. In this case, one can always pick $\lambda_n\in\Lambda_{u_1-\epsilon,u_1+\epsilon}$ such that $\lambda_n(u_1) = u_1\topp n$ and $\sup_{t\in[u_1-\epsilon, u_1+\epsilon]}|\lambda_n(t)-t| = |u_1\topp n-u_1|$. Thus, for all such choices of $\epsilon$, $n$ and $\lambda_n$,
\begin{multline*}
d_{u_1-\epsilon,u_1+\epsilon}(f\topp n,f) \leq |u_1\topp n - u_1| \vee \sup_{t\in[u_1-\epsilon,u_1+\epsilon]}|f_t - (f\topp n\circ\lambda_n)_t|\\
\leq |u_1\topp n-u_1|\vee\sup_{t_1,t_2\in[u_1-\epsilon,u_1)}|f_{t_1}-(f\topp n\circ\lambda_n)_{t_2}| \vee |j\topp n_1-j_1|.
\end{multline*}
The second term above equals $\sup_{t_1,t_2\in[u_1-\epsilon,u_1)}|m_{t_1}-(m\topp n\circ\lambda_n)_{t_2}|$. Now to show~\eqref{eq:r2} it remains to show that
\equh\label{eq:r2'}
\lim_{\epsilon\downarrow0}\limsupn\sup_{t_1,t_2\in[u_1-\epsilon,u_1)}|m_{t_1}-(m\topp n\circ\lambda_n)_{t_2}| = 0.
\eque
To see this, for any $\delta>0$ we take $\epsilon>0$ small enough so that $\sup_{t\in[u_1-\epsilon,u_1+\epsilon]}|m_t-m_{u_1}|\leq \delta/3$. At the same time, for $n$ large enough, $\sup_{n\in[0,1]}|m_t\topp n-m_t|\leq \delta/3$. Thus, for all $t_1,t_2\in[u_1-\epsilon,u_1)$,
\begin{multline*}
|m_{t_1}-(m\topp n\circ\lambda_n)_{t_2}|\\
\leq |m_{t_1}-m_{u_1}| + |(m\circ\lambda_n)_{t_2}-m_{u_1}|+|(m\topp n\circ\lambda_n)_{t_2}-(m\circ\lambda_n)_{t_2}|\leq\delta.
\end{multline*}
Therefore~\eqref{eq:r2'} follows.

We have thus completed the proof of~\eqref{eq:r2} in the case  $j_1>m_{u_1}$. The case $j_1<m_{u_1}$ is trivial since for $n$ large enough, $f\topp n = m\topp n\to m = f$ (this is the case that $j_1$ does not change the path $m$ at all). We only discuss the case $j_1=m_{u_1}$. In this case,~\eqref{eq:r1} still holds and we show
\[
\limn d_{u_1-\epsilon,1}(f\topp n,f) = 0.
\]
Since $f = m$ in this case, and~\eqref{eq:m} still holds, it suffices to show
\[
\limn d_{u_1-\epsilon,1}(f\topp n,m\topp n) = 0.
 \]
Indeed,  
\[
d_{u_1-\epsilon,1}(f\topp n,m\topp n)\leq \sup_{t\in[u_1-\epsilon,1]}|f\topp n_t-m\topp n_t| = \pp{j\topp n_1-m\topp n_{u_1\topp n}}\vee0,
\]
where the last step follows from~\eqref{eq:fnt}. 
The upper bound above goes to zero since $\limn j\topp n_1 = j_1 = m_{u_1} = \limn m_{u_1\topp n}\topp n$.

We have thus proved that $\Psi\topp 1$ is continuous. Now consider $\Psi\topp k$ for general $k\geq 2$. Recall that it is assumed that $\indn u$ are pairwise disjoint. In view of~\eqref{eq:bump}, for each $i$ such that $j_i\defe x_{u_i}+y_i>m_{u_i}$, and consider $v_i = 1\wedge \inf\{s\leq 1:m_s>j_i\}$. Then, $f$ can be seen as $m$ {\it lifted up to $j_i$} over interval $[u_i,v_i)$, and a jump is created at $u_i$. Let $I\subset\{1,\dots,k\}$ be the collection of all such $i$s. 
If $\{[u_i,v_i]\}_{i\in I}$ are all disjoint, then it suffices to divide the interval $[0,1]$ into $2|I|+1$ appropriate subintervals, and proceed as before. 

If there are overlaps among $\{[u_i,v_i]\}_{i\in I}$, then a new situation that is unseen in the case $k=1$ and needs to be dealt with is when for some $i,i'\in I$, $u_i<u_{i'}<v_i\leq v_{i'}$, and $j_i<j_{i'}$. This is the case that $m$ is lifted up at least twice over $[u_{i'},v_i)$. In this case, we can establish the convergence on intervals $[u_i-\epsilon,u_i+\epsilon], [u_i+\epsilon,u_{i'}-\epsilon], [u_{i'}-\epsilon,u_{i'}+\epsilon]$, for $\epsilon>0$ small enough, and others. The new types of intervals are $[u_i+\epsilon,u_{i'}-\epsilon]$ and $[u_{i'}-\epsilon,u_{i'}+\epsilon]$, although the convergence over them can be shown in a similar way as in the proof of~\eqref{eq:r3} and~\eqref{eq:r2}, respectively. The details are omitted.

At last, we remark that it is crucial here to restrict to the subset so that $\indn u$ are pairwise disjoint. Otherwise the jumps may be clustered and in this case, $J_1$-topology is too strong for establishing tightness. For an illustration of such an issue, see for example \citet{avram92weak} for a similar phenomena when establishing invariance principles for heavy-tailed processes.
\end{proof}

\section{Proof of Theorem~\ref{thm:1}, general perturbation}\label{sec:thm1negative}
In this section we prove the general case that $Y_i$ may take negative values. Recall that
\[
Z_{n,t} = \frac{M_{\floor{nt}}}{n^H} = \frac1{n^H}\max_{i=0,\dots,\floor{nt}}(S_i+Y_i).
\]
Introduce two modifications based on negative perturbations
\eqnh
Z_{n,t}^{-\infty} & = &  \frac1{n^H}\max_{i=0,\dots,\floor{nt}}(S_i+Y_i)\indd{Y_i\geq 0}\\
Z_{n,t}^0 & = &  \frac1{n^H}\max_{i=0,\dots,\floor{nt}}(S_i+Y_i\indd{Y_i\geq 0}).
\eqne
In this way, for all $n\in\N$,
\[
Z_{n,t}^{-\infty}\leq Z_{n,t}\leq Z_{n,t}^0, t\in[0,1].
\]
Intuitively, $Z_{n}^0$ corresponds to the maximum process of $\{S_n+Y_n\}_{n\in\N}$ with all negative perturbation $Y_i$ set to zero, and $Z_n^{-\infty}$ all negative perturbation $Y_i$ set to $-\infty$. 

Observe that we have proved that $Z_n^0\weakto Z^H$ in $D[0,1]$. This is because $\{Y_n\indd{Y_n\geq 0}\}_{n\in\N}$ are non-negative, i.i.d., and have the same upper tail distribution as $\indn Y$ in~\eqref{eq:RV}. Therefore, to complete the proof we show for all $\epsilon>0$,
\equh\label{eq:uniform}
\limn\proba\pp{\sup_{t\in[0,1]}Z_{n,t}^0-Z_{n,t}^{-\infty}>\epsilon} = 0.
\eque

To see this, we introduce some notation. Let $0=\tau_1<\tau_2<\cdots<\tau_K$ be the collection of all indices of non-negative $Y_i$s in increasing order, and set $\tau_{K+1} = n$. The key observation is the following:
\begin{multline*}
\ccbb{\sup_{t\in[0,1]}Z_{n,t}^0-Z_{n,t}^{-\infty}>\epsilon} \\
\subset \bcupp i1K\ccbb{\sup_{t\in\bb{\frac{\tau_i}n,\frac{\tau_{i+1}}n}}B_{n,t}^H - \pp{B_{n,\frac{\tau_i}n}^H+\frac{Y_{\tau_i}}{n^H}}>\epsilon}\\
\subset\bcupp i1K\ccbb{\sup_{t\in\bb{\frac{\tau_i}n,\frac{\tau_{i+1}}n}}B_{n,t}^H - B_{n,\frac{\tau_i}n}^H>\epsilon}.
\end{multline*}
Intuitively, if $Z_{n,t}^0-Z_{n,t}^{-\infty}>\epsilon$ for some $t\in(0,1)$, then it must be caused by certain large fluctuation of $B_n^H$. 

Now, let $\Delta_n$ be the size of the largest gap between $\tau_i$s, namely $\Delta_n = \max_{i=1,\dots,K}(\tau_{i+1}-\tau_i)$. We then have
\begin{multline*}
\ccbb{\sup_{t\in[0,1]}Z_{n,t}^0-Z_{n,t}^{-\infty}>\epsilon} \subset 
\ccbb{\sup_{\substack{s,t\in[0,1]\\|s-t|\leq\Delta_n/n}}B_{n,s}^H - B_{n,t}^H>\epsilon}\\
= \ccbb{\omega_{\Delta_n/n}(B_n^H)>\epsilon},
\end{multline*}
where  $\omega_\delta(x)$ is the uniform modulus of continuity defined as
\[
\omega_\delta(x) = \sup_{\substack{s,t\in[0,1]\\|s-t|\leq\delta}}|x_s-x_t|.
\]
We now make use of information of the path properties of $B_n^H$ implied by the invariance principle assumption~\eqref{eq:SnWIP}. By \citet[(12.9)]{billingsley99convergence}, 
\[
\omega_\delta(x) \leq 2\omega_\delta'(x) + J(x).
\]
Here,  $\omega'_\delta$ is the modulus of continuity commonly used for convergence in $D[0,1]$
\[
\omega_\delta'(x) = \inf\max_{i=1,\dots,v}\sup_{s,t\in[t_{i-1},t_i)}|x_s-x_t|,
\]
where the infimum is taken over all the sequences $0=t_0<t_1<\cdots<t_v=1$ such that $\min_{i=1,\dots,v}(t_i-t_{i-1})>\delta$, and $J$ is the maximal jump of process $x\in D[0,1]$ defined as
\[
J(x) = \sup_{0<t\leq 1}|x_t-x_{t^-}|.
\]
Therefore,
\equh\label{eq:uniform2}
\proba\pp{\sup_{t\in[0,1]}Z_{n,t}^0 - Z_{n,t}^{-\infty}>\epsilon} \leq \proba\pp{\omega_{\Delta_n/n}'(B_n^H)>\frac\epsilon4} + \proba\pp{J(B_n^H)>\frac\epsilon2}.
\eque
Since $B_n^H\weakto \B^H$ in $D[0,1]$ and $\B^H\in C[0,1]$ with probability one, it follows from \citet[Theorem 13.4]{billingsley99convergence}  that 
\equh\label{eq:J}
J(B_n^H)\weakto 0.
\eque 

Now to complete the proof we need a control on $\omega_{\Delta_n/n}'(B_n^H)$. 
On one hand,
observe that $\Delta_n-1$ is the length of the longest head run of $n$ independent coin tossings with head probability $p = \proba(Y_1<0)$ (see e.g.~\citep{gordon86extreme}), and in the current situation $p\in(0,1)$. Then, it is known that
\[
\ccbb{\Delta_n-1 - \log_{1/p}(n(1-p))}_{n\in\N}
\]
is tight \citep{arratia89two,gordon86extreme}. So for any $\delta>0$,
\[
\limsupn\proba\pp{\omega_{\Delta_n/n}'(B_n^H)>\epsilon/4}\leq\limsupn\proba\pp{\omega_\delta'(B_n^H)>\epsilon/4}.
\]
On the other hand, the invariance principle of $B_n^H$ implies, by \citep[Theorem 13.2]{billingsley99convergence}, for all $\epsilon>0$
\equh\label{eq:omega}
\lim_{\delta\downarrow0}\limsupn\proba\pp{\omega_\delta'(B_n^H)>\epsilon} = 0.
\eque
Therefore, combining~\eqref{eq:uniform2},~\eqref{eq:J} and~\eqref{eq:omega} yields the desired~\eqref{eq:uniform} and hence completes the proof.

\bibliographystyle{apalike}
\bibliography{references}

\end{document}